\documentclass[12pt]{amsart}

\usepackage{xypic,enumerate}

\usepackage{amsthm,amsmath,amssymb}

\usepackage{graphicx,  ctable, comment}

\usepackage[colorlinks=true,citecolor=black,linkcolor=black,urlcolor=blue]{hyperref}

\theoremstyle{plain}
\newtheorem{theorem}{Theorem}
\newtheorem{lemma}[theorem]{Lemma}
\newtheorem{corollary}[theorem]{Corollary}
\newtheorem{proposition}[theorem]{Proposition}

\theoremstyle{definition}
\newtheorem{definition}[theorem]{Definition}
\newtheorem{example}[theorem]{Example}

\theoremstyle{remark}
\newtheorem{remark}[theorem]{Remark}

\newcommand{\fm}{\mathfrak{m}}

\newcommand{\FF}{\mathbb{F}}
\newcommand{\GG}{\mathbb{G}}

\newcommand{\ZZ}{\mathbb{Z}}

\newcommand{\coker}{\operatorname{coker}}
\def\HH{\operatorname{H}}
\def\ext{\operatorname{Ext}}
\def\hm{\operatorname{Hom}}

\def\im{\operatorname{im}}
\def\rightarrownk{\operatorname{rank}}
\def\length{\operatorname{length}}
\numberwithin{equation}{section}

\begin{document}

\title{Filtrations of Totally Reflexive Modules}

\author[Denise A. Rangel Tracy]{Denise A. Rangel Tracy}

\address{ Department of Mathematics,  Syracuse University, Syracuse, NY USA}

\email{detracy@syr.edu}



\date{\today}

\keywords{totally reflexive, exact zero divisor}

\subjclass[2010]{Primary 13D99: 18G99}

\begin{abstract}

 In this paper, we will introduce a subcategory of totally reflexive modules that have a saturated filtration by other totally reflexive modules. We will prove these are precisely the totally reflexive modules with an upper-triangular presentation matrix. We conclude with an investigation of the ranks of  $\ext^1$ of two such modules over a specific ring.\end{abstract}

\maketitle
\thispagestyle{empty}

\section{Introduction and Preliminaries}

Totally reflexive modules were introduced by Auslander and Bridger \cite {AusBr} as modules of Gorenstein dimension zero. It was not until 2002 when  Avramov and Martsinkovsky \cite{AvMar} first referred to them as totally reflexive, to better emphasize their homological properties. Over Gorenstein rings, totally reflexive modules are exactly the maximal Cohen-Macaulay modules, whose representation theory is well developed, for example see \cite{LWbook}, \cite{Yosbook}. However, over non-Gorenstein rings much less is known. Something that is known from  \cite{sing} is that if there exists a non-free totally reflexive module over a non-Gorenstein local ring with the residue field of characteristic 0,  then there exist infinitely many non-isomorphic indecomposable totally reflexive modules over the ring as well. In fact in \cite {BT} and \cite {Holm} infinite families of non-isomorphic indecomposable totally reflexive modules are constructed, both of which arise from exact zero divisors. 

Although totally reflexive modules were originally defined for a broader class of rings, for this paper we assume $R$ to be a commutative Noetherian local ring. A finitely generated $R$-module $M$ is called \emph{totally reflexive} if the biduality map, $\delta: M \rightarrow \hm_R(\hm_R(M,R), R)$ is an isomorphism, $\ext_R^i(M,R) =0 $ and  $\ext_R^i(\hm_R(M,R), R) =0 $  for all $i>0.$ 
Projective modules are obviously totally reflexive. We call a nonzero totally reflexive module \emph{nontrivial} if it is not a projective module.  A complex is called \emph {acyclic} if its homology is zero.  A \emph{totally acyclic complex} is an acyclic complex $\mathbb{A}$ whose \emph{dual} $\hm_R(\mathbb{A}, R)$ is also acyclic. A module being totally reflexive is equivalent to being a syzygy in a totally acyclic complex of finitely generated free modules. Recall that over a local ring $(R, \fm)$ a complex $(\FF, \partial)$ of free modules is said to be \emph{minimal} if $\im \partial_i^{\FF} \subseteq \fm\FF_{i-1}$ for all $i.$ 

Through out this paper we will investigate these modules via their presentation matrix, which is a matrix whose columns are a minimal generating set of the module. Presentation matrices are not unique. However, when over commutative Noetherian local rings two minimal presentation matrices of the same module must be equivalent, see \cite{equivmat}.

In this paper, we investigate the structure of totally reflexive modules over local non-Gorenstein Artinian rings by looking at their totally reflexive submodules. We are most interested in the case when the quotient module formed by the totally reflexive module and a totally reflexive submodule is also totally reflexive and of minimal length. If this occurs, then we say that the module has a \emph{saturated TR-filtration,} see Definition \ref{TRfiltDefn} for a more precise description. We prove that knowing if a totally reflexive module has a saturated TR-filtration is directly linked to the existence of a minimal presentation matrix of the module that is upper triangular. \\

\begin{textbf}{Theorem}\end{textbf}. \emph{Let $(R,\fm)$ be a non-Gorenstein ring with $\fm^3=0\neq \fm^2 $ that contains exact zero divisors, and suppose that $T$ is a totally reflexive $R$-module. There exists a saturated TR-filtration of $T$ if and only if $T$ has an upper triangular presentation matrix.\\}

From this theorem, we obtain two corollaries. Corollary \ref{cor1} gives further information about the form of this upper triangular presentation matrix. Additionally, Corollary \ref{UTres} proves the existence of a complete resolution in which every differential can be simultaneously represented by upper triangular matrices. We conclude with an extensive example of the theory, including a study of the number of ways a saturated TR-filtration can occur.

\subsection{Previously Known Results}
Since the results in this paper are proven with the ring being commutative local non-Gorenstein with the cube of the maximal ideal equaling zero, we give two previously known theorems which provide useful facts about these types of rings and totally reflexive modules over them. For a local Artinian ring $(R, \fm, k),$ set $e=\length (\fm/\fm^2)$ to be the embedding dimension of $R.$ 

\begin{theorem} \cite[Theorem 3.1]{Yos} \label{YosThm}
Let $(R,\fm,k)$ be a commutative local Artinian non-Gorenstein ring with $\fm^3=0\neq\fm^2.$ If there exist a nontrivial totally reflexive $R$-module $T,$ then the following conditions hold:
\begin{enumerate}[(a)]
\item $(0 :_R \fm)=\fm^2$
\item $\length (0 :_R \fm)=e-1,$ in particular, $\length(R)=2e.$ \label{YosThmlength}
\item Let $n$ be the minimal number of generators of $T,$ then $\length (T)=ne.$
\item $T$ has a  minimal free resolution of the form \label{YosThmReso}
$$\cdots \rightarrow R^n \xrightarrow{d_2} R^n \xrightarrow{d_1}R^n \rightarrow T \rightarrow 0.$$
\end{enumerate}
\end{theorem}

 \begin{remark}
Part (\ref{YosThmReso}) of the previous theorem implies that every differential in a minimal free resolution of $T$ can be presented by a square matrix. Moreover, any minimal presentation matrix of $T$ is a square matrix. This fact holds for a complete resolution of $T$ as well, since any syzygy in a complete resolution of a totally reflexive modules is also totally reflexive.
\end{remark}

\begin{theorem}\cite[Theorem 5.3]{BT} \label{BTcite}
 Let $(R, \fm)$ be a local ring with $\fm^3 = 0$ and $e  \geq 3.$ Let $x$ be an element of $\fm/\fm^2;$ the following conditions are equivalent.
\begin{enumerate} [(i)]
\item The element $x$ is an exact zero divisor in $R$.
\item  The Hilbert series of $R$ is $1+e\tau +(e-1)\tau^2,$ and there exists an exact sequence
of finitely generated free R-modules
$$ \FF: \qquad  F_3\rightarrow F_2 \xrightarrow{} F_1 \xrightarrow{\psi}F_0 \rightarrow F_{-1} $$
such that $\hm_R(\FF,R)$ is exact, the homomorphisms are represented by matrices with entries in $\fm,$ and $\psi$ is represented by a matrix in which some row has $x$ as an entry and no other entry from $\fm/\fm^2.$
 \end{enumerate}
\end{theorem}
An exact zero divisor is a special type of ring element, one that is quite significant in these results, see Definition \ref {defnexpair}.


\section {Filtrations and Upper Triangular Presentation Matrices}
Before we present the main theorem of this paper, we need to make note of a few facts when considering modules via their presentation matrices. 

\begin{lemma} \label{copyk} Let $(R,\fm,k)$ be a commutative local non-Gorenstein ring with $\fm^3=0\neq\fm^2$ and $M$ an $R$-module. If a presentation matrix  of $M$ has a column whose entries are contained in $\fm^2,$ then the first syzygy of $M$ contains a copy of $k$ as a direct summand. 
\end{lemma}
\begin{proof}
Let $\fm=(x_1, \dots , x_e),$ where $e$ is the embedding dimension of $R,$ and suppose $M$ has the presentation matrix  $\mathbf{M}= [c_1, \dots, c_s],$ where $c_i \in R^{b_0}.$ Also, after row and column operations, we may assume that $c_1 \subset \fm^2 R^{b_0}$. Let 
$$\FF: \qquad \cdots \rightarrow R^{b_2}\xrightarrow{\mathbf{N}}  R^{b_1}\xrightarrow{\mathbf{M}} R^{b_0}\rightarrow 0$$
be a deleted free resolution of $M$. Since $\mathbf{MN}=0,$ for $1\leq i\leq e$ we have that the elements $\left(\begin{array}{c} x_i\\ 0\\ \vdots \\ 0\end{array} \right),$  are part of a minimal generating set of the syzygies of $\mathbf{M}.$ Without loss of generality, assume that $\mathbf{N}$ has the form
$$\left[\begin{array}{cccccc}
 x_1 &\ldots & x_e& 0& \ldots &0\\
0    &\ldots & 0 & *& \ldots &*\\																	
\vdots& \ddots& \vdots&\vdots& \ddots& \vdots\\
0    &\ldots& 0 & *& \ldots &*\\																	
\end{array} \right]. $$

Note that the entries in the first row, after the $e$th column, are all zeroes. This is true since if there was a nonzero entry there, then it would be a linear combination of the first $e$ columns. Therefore, $\coker \mathbf{N} \cong k \oplus X,$ for some $R$-module $X$.
\end{proof}

\begin{proposition}\label{ksumcor} Let $(R,\fm,k)$ be a 
 non-Gorenstein ring with $\fm^3=0\neq\fm^2.$ If the $n$th syzygy of an $R$-module $M$,  $\Omega^R_n(M)$  contains a copy of $k$ as a direct sum, then $M$ is not totally reflexive.

\end{proposition}

\begin{proof}
Assume that for some $R$-module $X,$ we have $\Omega^R_n(M)\cong k \oplus X$. Now suppose that $M$ is totally reflexive, and therefore, $\Omega^R_n(M)$ is totally reflexive as well since it is a syzygy of $M$. Then, for all $i>0,$ we have $\ext^i_R(\Omega^R_n(M),R)=0 $ and thus $\ext^i_R(k\oplus X,R)=0$ for all $i>0,$ which implies that $\ext^i_R(k,R)=0,$ for all $i>0$. This holds if and only if $R$ is Gorenstein. However, we assumed $R$ not to be Gorenstein, and therefore $M$ cannot be totally reflexive.
\end{proof}

\begin{definition}\label{TRfiltDefn} For a totally reflexive $R$-module $T$, a \emph{TR-filtration} of $T$ is a chain of submodules 
$$0 = T_0 \subset T_1 \subset \dots \subset T_{n-1} \subset T_n = T,$$
in which the following hold for all  $i=1, \dots, n:$
\begin{enumerate}[(i)]
\item $T_i$ is totally reflexive
\item $T_{i}/T_{i-1}$ is totally reflexive
\item $T_{i}/T_{i-1}$ contains no proper nonzero totally reflexive submodules.\\

\noindent We say a  filtration is a \emph{saturated TR-filtration} if, in addition, we have that 
\item $T_{i}/T_{i-1}$ is of minimal length among the totally reflexive $R$-modules.  
\end{enumerate}

\end{definition}
In \cite{HenSega} the elements of the ring which play a vital role in the construction of these filtrations are defined as follows.

\begin{definition}\label{defnexpair}
For a commutative ring $A$, a non-unit $a \in A$ is said to be an \emph {exact zero divisor} if there exists $b\in A$ such that $(0: a)= (b)$  and $(0: b)= (a).$ If $A$ is local then $b$ is unique up to unit, and we call $(a,b)$ an \emph {exact pair of  zero divisors}.\\

This is equivalent to the existence of a free resolution of $A/(a)$ of the form
$$ \cdots \rightarrow A \xrightarrow{[b]}A \xrightarrow{[a]}A \xrightarrow{[b]}A \xrightarrow{[a]}A \rightarrow 0$$

\begin{remark} \label{TAexact}
The complex in Definition \ref{defnexpair} is  totally acyclic and thus the modules $A/(a)$ and $A/(b)$ are totally reflexive. In fact, in certain cases the reverse it also true.

\end{remark}

\end{definition} 

\begin{lemma} \label{expairisTR} Let $(R, \fm)$ be a local ring with $\fm^3=0 \neq\fm^2,$ and let $a\in R.$ If $R/(a)$ is a totally reflexive module, then $a$ is an exact zero divisor. 

\end{lemma}

\begin{proof}
Let $R/(a)$ be a totally reflexive module, and so by Theorem \ref{BTthm} it has a free resolution of the form
$$\FF: \qquad \cdots \rightarrow R \xrightarrow{b_2}R \xrightarrow{b_1}R \xrightarrow{a} R \rightarrow 0.$$
This implies that $(0:a)=(b_1).$ Apply the $\hm$ functor, we obtain the complex
$$\hm_R(\FF,R): 0 \rightarrow \hm_R(R/(a),R) \xrightarrow{}R \xrightarrow{a^*}R \xrightarrow{b^*_1} R \rightarrow \cdots,$$
 which is  exact because $R/(a)$ is totally reflexive. This is isomorphic to 
$$ 0 \rightarrow (R/(a))^* \xrightarrow{}R \xrightarrow{a}R \xrightarrow{b_1} R \rightarrow \cdots.$$
Therefore, $\ker(a)=(b_1)$ and we have that $(0:b_1)=(a).$  Hence $(a,b_1)$ are an exact pair of zero divisors.
\end{proof}

\begin{lemma}\label {lem1} Let $(R,\fm)$ be a 
non-Gorenstein ring with $\fm^3=0\neq\fm^2$ and embedding dimension $e.$  If a totally reflexive $R$-module $T$ has a minimal presentation matrix  that contains a row with only one nonzero entry, then there exists a totally reflexive submodule $U \subset T$ such that length $(U) = $ length $(T) -e.$ \\
\end{lemma}

\begin{proof}
 Let 
$$\FF: \qquad \cdots \rightarrow F_2\xrightarrow{\mathbf{W}} F_1 \xrightarrow{\mathbf{T}} F_0 \rightarrow F_{-1} \rightarrow \cdots$$
be a complete free resolution of $T,$ where $\mathbf{W}=(\omega_{ij})$ is a presentation matrix of $\Omega^R_1(T)$ and let
 $$\mathbf{T}=\left[ \begin{array}{cccc}
t_{11} &\cdots &     & t_{1n} \\
\vdots  &          &     &\vdots\\
t_{n-1 1} &\cdots &      & t_{n-1 n}\\
0&      \cdots     &0  &  t_{nn}\\
\end{array} \right] $$

\noindent be a presentation matrix  of $T.$
From Theorem  \ref{BTcite}  we know that $t_{nn}$ is an exact zero divisor of $R.$ Let $s$ generate $(0:t_{nn}).$ Since $\mathbf{TW}=0,$ we have that $t_{nn}\omega_{n j}=0 $ for all $j=1, \dots , n$. Thus, the ideal $(\omega_{n1}, \dots , \omega_{nn}) \subset (0: t_{nn})=(s)$. This, with Lemma \ref{copyk}, implies that for some $i$ we have $\omega_{ni} = s,$ up to units. Therefore, every entry in the $n$th row of $\mathbf{W}$ is a multiple of $s$. We can apply column operations to $\mathbf{W}$ to assume that
$$ \mathbf{W} = \left[ \begin{array}{cccc}
w_{11} &\cdots && w_{1n} \\
\vdots&&& \vdots\\
0& \cdots & 0 & s
\end{array} \right]$$
 is another presentation matrix of $\Omega^R_1(T)$. 
 
 Define $\textbf{W'} $ to be the $(n-1) $ by $(n-1)$  matrix obtained  by deleting the $n$th row and  $n$th column from $\textbf{W}. $ Similarly, we define $ \textbf{T'} $ to be the $(n-1) $ by $(n-1)$  matrix obtained  by deleting the $n$th row and  $n$th column from $\textbf{T}. $  Note that $\mathbf{W'T'}=0,$ and now consider the following commutative diagram: 

\[
\xymatrixrowsep{1.7pc}\xymatrixcolsep{2.1pc}\xymatrix{
 & 0 \ar@{->}[d]^{\qquad}& &  0 \ar@{->}[d]^{\qquad}&& 0 \ar@{->}[d] ^{\qquad}&\\
\GG':& R^{n-1 } \ar@{->}[rr]^{ \mathbf{W'}} & &  R^{n-1} \ar@{->}[rr]^{ \mathbf{T'}} & &  R^{n-1} \ar@{->}[r] &0\\
\\
\GG:&R^n  \ar@{->}[rr]^{\mathbf{W}} \ar@{<-}[uu]^q  && R^n \ar@{->}[rr]^{\mathbf{T}} \ar@{<-}[uu]^q  && R^n \ar@{->}[r] \ar@{<-}[uu]^q &0\\
\\
\GG'':&R \ar@{->}[rr]^{[w]} \ar@{<-}[uu]^p  && R \ar@{->}[rr]^{[t_{nn}]} \ar@{<-}[uu]^p  && R \ar@{->}[r] \ar@{<-}[uu]^p&0\\
                & 0 \ar@{<-}[u]                        & &  0 \ar@{<-}[u]                            && 0 \ar@{<-}[u] &\\
                &&&&&&\\
}
\]

\noindent where $q:=\left[ \begin{array}{ccc} 1& & 0\\ &\ddots & \\ 0& &1\\ 0& \dots &0 \end{array} \right]$ and $p:= \left[\begin{array}{cccc} 0 & \ldots &0 & 1\end {array}\right].$ Thus we have an exact sequence of complexes
$$ 0 \rightarrow \GG ' \rightarrow \GG \rightarrow \GG '' \rightarrow 0$$  
which yields the following long exact sequence of homology.  
$$ \cdots \rightarrow \HH_1(\GG '') \rightarrow \HH_0(\GG ') \rightarrow \HH_0(\GG) \rightarrow \HH_0(\GG '') \rightarrow 0.$$

\noindent Note that $\HH_1(\GG '') =0$ since $ R\stackrel{\left[ w \right]}{\longrightarrow} R \stackrel{\left[t_{nn}\right]}{\longrightarrow}  R$ is exact. 
 Let $\coker \mathbf{T'} $ and so  $ U \subset T$. Therefore,  $\HH_0(\GG '')=R/ (t_{nn}) \cong T/U$ and we have the short exact sequence 
$$ 0 \rightarrow U \rightarrow T \rightarrow T/U \rightarrow 0. $$

 To see that $U$ is in fact totally reflexive, note that $T/U \cong R/(t_{nn})$ is totally reflexive since $(t_{nn})$ is an exact zero divisor. This, along with the fact that $T$ is totally reflexive, implies that $U$ is as well, see \cite[Corollary 4.3.5]{GorDim}. From Yoshino's theorem, here listed as Theorem \ref{YosThm}, part (\ref{YosThmlength}), we know that the length of $R/ (t_{nn})$ is the number of its minimal generators  times the embedding dimension. Therefore, $\length(R/ (t_{nn}))=e$ and thus $\length (U)=\length (T) -e.$
\end{proof}

\begin{theorem}\label {thm1} Let $(R,\fm)$ be a non-Gorenstein ring with $\fm^3=0\neq \fm^2 $ that contains exact zero divisors, and suppose that $T$ is a totally reflexive $R$-module. There exists a saturated TR-filtration of $T$ if and only if $T$ has an upper triangular presentation matrix.
\end{theorem}

\begin{proof}
Let $T$ be totally reflexive $R$-module with minimal number of generators $\mu(T)=n$ and suppose there exists a saturated TR-filtering of $T$. To show that $T$ has an upper triangular presentation matrix we will use induction on $\mu(T)$. If $\mu(T)=1,$ then any presentation matrix  of $T$ would be of size $ 1 \times 1$ and thus trivially upper triangular.

 Assuming true for $\mu(T)=n,$  consider the case $\mu(T)=n+1$. By Lemma \ref{lem1} there exists a totally reflexive $R$-module $M \subset T$ such that the sequence 
$$ 0 \rightarrow M \rightarrow T \rightarrow T/M \rightarrow 0$$
is exact and $T/M$ is cyclic. Define $N:= T/M$ and let  $\FF $ and $\GG$  be free resolutions of $M$ and $N$ respectively. Since $R$ is local with $\fm^3=0,$ we can assume that $F_i=R^n$  and $G_i=R,$ for all $i\geq 0$, \cite{Acyclic}.  By applying the Horseshoe Lemma, we get the following diagram:
\begin{equation}\label{preset}
\xymatrixrowsep{3pc}\xymatrixcolsep{3pc}\xymatrix{
0\ar@{->}[r] &R^n \ar@{->}[r] \ar@{->}[d]^{d_1^M}& R^n\oplus R \ar@{->}[r]\ar@ {->}[d]^{d_1^T} & R\ar@ {->}[r]\ar@ {->}[d]^{d_1^{N}}\ar@ {->}[dl]^{\sigma_1}& 0\\
0\ar@{->}[r] &R^n \ar@{->}[r] \ar@{->}[d]^{\epsilon_M}& R^n\oplus R \ar@{->}[r]\ar@ {->}[d]^{\epsilon_T} & R\ar@ {->}[r]\ar@ {->}[d]^{\epsilon_N}\ar@ {->}[dl]^{\sigma_0}& 0\\
0 \ar@{->}[r] & M  \ar@{->}[r] & T \ar@{->}[r]& N \ar@{->}[r]&0\\
&0\ar@{<-}[u]&0\ar@{<-}[u]&0\ar@{<-}[u]&
}
\end{equation}

\noindent where $d_1^T (f,g) \mapsto d_1^M(f) + \sigma_1(g)$. Since $d_1^M(f) \in R^n$ and  $\sigma_1 (g) \in R^n\oplus R,$  when we consider them as columns in a presentation matrix of $T$ we have the matrix
\begin{equation}\label{prest} \left[\begin{array}{cc} d_1^M& f'\\ 0 & g' \end{array} \right], \end{equation}
 where $\sigma_1(g)=f'+g'$ for some $f' \in R^n$ and $g' \in R,$ we see that it is upper triangular. This matrix is a presentation matrix of $T$. By induction, the matrix representing $d_1^M(f)$ can be taken to be upper triangular and therefore (\ref{prest}) is upper triangular. \\

Now suppose $T$ has an upper triangular presentation matrix, say
$$\left[ \begin{array}{ccc}
t_{11} &\cdots      & t_{1n} \\
 & \ddots         &\vdots\\
0&       &  t_{nn}
\end{array} \right] .$$
Again we will use induction on the minimal number of generators.  If $n=2,$ then
by Lemma \ref{lem1} there exists  a totally reflexive $R$-module $T_1\cong R/(t_{11})$ which is a submodule of $T.$ Thus, $T$ has a saturated TR-filtration  
 $$0=T_0 \subset T_1  \subset T_{2}=T$$
 with $T_2/T_1\cong R/(t_{22}).$
 Now assume that an $n$-generated totally reflexive module with an upper triangular presentation matrix has a saturated TR-filtration  
 $$0=T_0 \subset T_1 \subset \dots. \subset T_{n-1} \subset T_{n}$$
 with $T_i/T_{i-1} $ being one generated and totally reflexive for $i=1,\dots, n.$
 Consider an $(n+1)$-generated totally reflexive module $T,$ which has with an upper triangular presentation matrix. By Lemma \ref{lem1}, there exists a totally reflexive module $T_{n}$ such that $T_n$ is a submodule of $T$ and has  presentation matrix of the form
 $$\left[ \begin{array}{ccc}
 t_{11} &\cdots      & t_{1 n} \\
 & \ddots         &\vdots\\
0&       &  t_{n n}
\end{array} \right]$$
with $T/T_n \cong R/t_{n+1\,n+1}.$  This, combined with the induction hypothesis, shows that $T$ has a saturated TR-filtration.
\end{proof}

\begin{definition}
An \emph{upper triangular complex} is a complex in which every differential can simultaneously be represented by a square matrix $(a_{ij})$ such that $a_{ij}=0$ when $i>j.$

\end{definition}

\begin{corollary}\label{UTres} Let $(R,\fm)$ be a non-Gorenstein ring with $\fm^3=0\neq \fm^2 $ that contains exact zero divisors. If $T$ is a totally reflexive $R$-module that has an upper triangular minimal presentation matrix, then it has an upper triangular minimal complete resolution. 
\end{corollary}
\begin{proof} This will be done by induction on the number of minimal generators. Let $T$ be a totally reflexive $R$-module that has an upper triangular minimal presentation matrix with $\mu(T)=n.$ and thus the matrix is of size $n \times n.$ Let $\FF$ be a  free resolution of $T.$  By \cite[Theorem B]{Acyclic}, we have that $\rightarrownk_R F_1=n$ for all $i \in \ZZ.$ If $\mu(T)=1,$ then the free resolution of $T$ is trivially upper triangular. \\

Assume for $\mu(T)=n$ that $\FF $ is an upper triangular minimal free resolution and consider the case when $T$ has a presentation matrix of the form
$$\left[ \begin{array}{ccc}
 t_{11} &\cdots      & t_{1 \,n+1} \\
 & \ddots         &\vdots\\
0&       &  t_{n+1 \, n+1}
\end{array} \right]. $$
By Theorem \ref{thm1}, there exists an saturated TR-filtration and thus the short exact sequence
$$ 0 \rightarrow T_{n} \rightarrow T \rightarrow R/(t_{n+1\, n+1})\rightarrow 0.$$ 

From the Horseshoe Lemma, a diagram similar to (\ref{preset}) can be obtained and extended to include the first syzygies. By a similar argument to the proof of Theorem \ref{thm1}, we see that the first syzygy of $T$ has an upper triangular presentation matrix. Finally, any syzygy in a free resolution of totally reflexive module is also totally reflexive \cite{GorDim}. Therefore, we can apply this corollary to the $n$th syzygy to see that the $n+1$ the syzygy has an upper triangular presentation matrix. 
\end{proof}

\begin{corollary} \label{cor1} Let $(R,\fm)$ be a non-Gorenstein ring with $\fm^3=0\neq \fm^2 $ that contains exact zero divisors. For an $R$-module $M$ that has an upper triangular presentation matrix, the non-zero entries on the main diagonal of that presentation matrix are exact zero divisors if and only if $M$ is totally reflexive.
\end{corollary}

\begin{proof}
Both directions of this proof rely on the existence of a saturated TR-filtration. First, suppose $M$ is totally reflexive and has an upper triangular presentation matrix. Hence, we have a saturated TR-filtration
$$0 = M_0 \subset M_1 \subset \dots \subset M_{n-1} \subset M_n = M,$$ 
as well as short exact sequences
$$0\rightarrow M_{i-1} \rightarrow M_i\rightarrow M_{i}/M_{i-1}\rightarrow 0$$
for all $i=2, \cdots, n.$  If $i=2,$ then for some $u, v, \alpha \in R$ we have
$$0\rightarrow R/(u) \rightarrow \coker \left[\begin{array}{cc} u& \alpha \\ 0&v \end {array}\right]  \rightarrow R/(v)\rightarrow 0.$$ 
From Lemma \ref{expairisTR} both $u$ and $v$ are exact zero divisors. Assume this holds for $n-1,$ and consider
$$0\rightarrow M_{n-1} \rightarrow M_n\rightarrow M_{n}/M_{n-1}\rightarrow 0.$$
By the induction hypothesis, $M_{n-1}$ has a upper triangular presentation matrix with exact zero divisors on the main diagonal. If $\mathbf{M}=(a_{ij})$ is a presentation matrix of $M_n,$ then $M_{n}/M_{n-1}\cong R/(a_{nn}).$  Since $R/(a_{nn})$ is totally reflexive, $a_{nn}$ is an exact zero divisor. Therefore, for all $i,$ $a_{ii}$ are exact zero divisors.

 Now suppose $M$  has an upper triangular presentation matrix where the non-zero entries on the main diagonal are exact zero divisors. By the inductive part of the proof of Theorem \ref{thm1}, $M$ is totally reflexive.
\end{proof}


\section{Filtrations and Yoneda $\ext$}

Recall that the Yoneda definition \cite[Appendix A3]{Eisenbud} of $\ext_R^1(N,M)$  for two $R$-modules $M, N$ gives a correspondence between the elements of $\ext_R^1(N,M)$ and equivalence classes of exact sequences of the form
$$ 0 \rightarrow M  \rightarrow X  \rightarrow N \rightarrow 0.$$
If an $R$-module $T$ has a TR-filtration, 
$$0 = T_0 \subset T_1 \subset \dots \subset T_{n-1} \subset T_n = T,$$
 then there exist short exact sequences
$$0\rightarrow T_{i-1} \rightarrow T_i \rightarrow T_i/T_{i-1} \rightarrow 0$$
and hence an element in $\ext_R^1(T_i/T_{i-1},T_{i-1}).$  In this section, we investigate the number of possible nonequivalent short exact sequences of the above form for a specific ring. We start with  brief discussion on the Yoneda definition of $\ext^1.$

 Let 
$$\cdots \xrightarrow{} F_2 \xrightarrow{\partial_2}F_1\xrightarrow{\partial_1} F_0\xrightarrow{} N \xrightarrow{} 0$$
be a free resolution of $N,$  and so we have an exact sequence
$$ 0 \rightarrow \im\partial_1  \rightarrow F_0  \rightarrow N \rightarrow 0.$$
If we have some $(a +\im\partial^*_1 )\in \ext_R^1(N,M),$ then $a \in \hm(F_0, M).$ Define $a '$ to be $a$ restricted to $\im \partial_1$ and $\iota$ to be the natural inclusion from $\im \partial_1$  to $F_0.$ We can obtain the following commutative diagram using the pushout of $a'$ and $\iota.$

\[
\xymatrixrowsep{2.7pc}\xymatrixcolsep{1.8pc}\xymatrix{
  0 \ar@{->}[r]^{\qquad} &\im \partial_1 \ar@{->}[d]^{a'} \ar@{->}[r]^{\quad\iota\quad} &F_0 \ar@{->}[r]\ar@{->}[d] &  N \ar@{=}[d] \ar@{->}[r]^{\qquad}& 0 \\
  0 \ar@{->}[r] &  M\ar@{->}[r]^{\qquad} & (M\oplus F_0)/I \ar@{->}[r] &  N \ar@{->}[r]& 0 \\
}\]

\noindent where $I=(-a'(r), \iota(r) | r \in \im \partial_1).$ When considering filtrations, $(M \oplus R^{b_0})/I$ is the next larger submodule on the chain of submodules. In order to gain a better understanding of how this translates to a presentation matrix, we will compute this for the second module in a filtration over a specific ring.

For the reminder of this paper, let $S=k[X, Y, Z]/(X^2, Y^2, Z^2, YZ),$ and define $x, y, $ and $z$ to be the image of $X, Y, $ and $Z,$ respectively. Also, let $T_1=S/(x+by+cz)$ and $N=S/(x+dy+fz)$ where $b, c, d, f \in k,$  so they are the cokernal of an exact zero divisor. When considering $\ext_S^1(N, T_{1})$ by the Yoneda definition,  a non-zero element of $\ext_S^1(N, T_{1})$ corresponded to a short exact sequence of the form
$$0 \rightarrow T_1 \rightarrow T_2 \rightarrow N \rightarrow 0.$$
If $\ext_S^1(N, T_{1})\neq 0,$ then this will allow us to have a saturated TR-filtration of $T_2,$ 
$$0 \subset T_1 \subset T_{2}.$$
We now will find a presentation matrix for possible $T_2.$

Let
$$\FF_{N}: \qquad \cdots \xrightarrow{} S \xrightarrow{\partial_3} S\xrightarrow{\partial_2} S \xrightarrow{\partial_1} S \rightarrow 0$$
be a (deleted) minimal free resolution of $N$ where 
$$\partial_i = \left\{\begin{array} {l l} [x+dy+fz], & \quad \mbox{ if } i \mbox{ is odd} \\
\mbox{$[x-dy-fz],$} & \quad \mbox{ if } i \mbox{ is even. }\\ \end{array} \right.$$ 

For $(a +\im\partial^*_1) \in \ext_S^1(N, T_1)$ where $\alpha \in \ker(\hm_S(\partial_2,T_1)),$ we have that $\alpha \in \hm_S(S, T_1) \cong T_1.$ We will identify elements of $\hm_S(S, T_1)$ with $T_1$ under the natural isomorphism. Using the previous description,

\[
\xymatrixrowsep{2.7pc}\xymatrixcolsep{1.8pc}\xymatrix{
  0 \ar@{->}[r]^{\qquad} &\im \partial_1 \ar@{->}[d]^{\alpha'} \ar@{->}[r]^{\quad\iota\quad} &S\ar@{->}[r]\ar@{->}[d] &  N \ar@{=}[d] \ar@{->}[r]^{\qquad}& 0 \\
  0 \ar@{->}[r] &  T_1\ar@{->}[r]^{\qquad} & (T_1\oplus S)/I \ar@{->}[r] & N\ar@{->}[r]& 0 \\
}\]
 we find that the push out is $T_2:=(T_1 \oplus S)/I.$  To find a presentation matrix for $T_2,$ we need to find a matrix $\mathbf{T_2}$ in which the following complex is exact 
$$  S^2 \xrightarrow{ \mathbf{T_2} }S^2 \xrightarrow{p} T_2=(T_1 \oplus S)/I \rightarrow 0.$$

 Take $e_1, e_2$ to be the standard basis in $S^2$ and define $p(e_1)= (\overline{1}, 0) + I$ and $p(e_2)= (0, 1) + I$. To find the image of $\mathbf{T_2}$ we just need to find the kernel of $p$ which is shown below. 
 $$(x+by+cz) p(e_1) \equiv 0 \qquad \mbox{ and } \qquad -\alpha p(e_1) + (x+dy+fz) p(e_2) \equiv 0$$
 where $\alpha \in S/(x+by+cz)$. Therefore, a presentation matrix for  $T_2$ is
 $$\mathbf{T_2}=\left[\begin{array}{c c} x+by+cz & -\tilde{\alpha}\\ 0 & x+dy+fz \end {array} \right],$$
where $\tilde{\alpha}$ is a preimage of $\alpha$ in $S.$


\subsection{The Rank of $\ext_S^1(N,T_1)$ }

Since every element in $\ext^1$ can be viewed as a short exact sequence, to get a sense of how many nonequivalent sequences exist we study the $k$-vector space rank of $\ext^1.$ As before, let $T_1=S/ (x+by+cz)$ and $N=S/ (x+dy+fz).$  If
$$\FF_N: \qquad \rightarrow S \xrightarrow{\partial_2} S  \xrightarrow{\partial_1} S \rightarrow 0$$
is a (deleted) free resolution of $N$, then  the complex $\hm_S(\FF_{N}, T_1)$ has the form
$$ \qquad 0 \rightarrow T_1 \xrightarrow{\hm_S(\partial_1, T_1)} T_1\xrightarrow{\hm_S(\partial_2, T_1)} T_1 \xrightarrow{\hm_S(\partial_3, T_1)]} T_1  \rightarrow \cdots$$
 and $\ext_S^1(N, T_1)= \HH^1(\hm_S(\FF_{N}, T_1))=\ker(\hm_S(\partial_2, T_1))/\im (\hm_S(\partial_1, T_1)).$
We can compute the kernels and images of these maps:

\begin{eqnarray*}\ker (\hm_S(\partial_2, T_1)) &=& \left\{ \begin {array}{l l}
(\overline{1}) \cong T_1, \quad & \mbox{ if } b=-d \mbox{ and } c=-f\\
(\overline{y}, \overline{z}), \quad & \mbox{ otherwise}\\
\end{array} \right.\\
& & \\
\im (\hm_S(\partial_1, T_1))&=&((b-d)\overline{y}+ (c-f) \overline{z}).
\end{eqnarray*}

\noindent Therefore, 
$$\ext_S^1(N, T_1)=\left\{ \begin {array}{l l}
T_1/((d-b)\overline{y}+ (f-c) \overline{z}), \quad & \mbox{ if } b=-d \mbox{ and } c=-f\\
(\overline{y}, \overline{z})/((d-b)\overline{y}+ (f-c) \overline{z}), \quad & \mbox{ otherwise,}\\
\end{array} \right. $$
and, provided char$(k)\neq 2,$ we have

$$\rightarrownk (\ext_S^1(N, T_1))=\left\{ \begin {array}{l l}
3, \quad & \mbox{ if } b=c=d=f=0 \\
2, \quad & \mbox{ if } b=-d \mbox{ and } c=-f, \mbox{or if } b=d \mbox{ and } c=f\\
1, \quad & \mbox{ otherwise.}\\
\end{array} \right. $$

Suppose $(u, v)$ are an exact pair of zero divisors. Then for some $f_1,f_2 \in k$,  we have $u=x+f_1y+f_2z$ and $v=x-f_1y-f_2z.$  From the above computations, for some $[\alpha] \in \ext_S^1(S/u,S/v)$ and  some $S$-module $M,$ there exists the short exact sequence
$$\lambda:0 \rightarrow S/v \rightarrow M \rightarrow S/u \rightarrow 0,$$	
where $M$ has a presentation matrix 
 $$\mathbf{M}=\left[\begin{array}{c c} v & -\alpha\\ 0 & u \end {array} \right].$$
In this case, $1_S$ is a possibility for $-\alpha$  since $\overline{1} \in \ker(\hm_S(\partial_2, S/v)).$ However
 $$\coker \mathbf{M}= \coker \left[\begin{array}{c c} v & 1_S\\ 0 & u \end {array} \right] \cong S.$$
This implies that $\lambda$ represents part of the complete resolution of $R/v$. Since this is always the case when we have a pair of exact zero divisors, we exclude it to focus on when the non-syzygy cases occurs. Hence, we define $\Gamma_S(N, T_1)$  to represent the rank of the elements in $\ext_S^1(N, T_1)$ which are not part of a complete resolution of $T_1.$ Therefore, \\

\begin {equation} \label{rank} \Gamma_S(N, T_1) :=\left\{ \begin {array}{l l}
2 \quad & \mbox{ if } b=c=d=f=0  \mbox{ or } b=d \mbox{ and } c=f, \\
1 \quad & \mbox{ otherwise.}\\
\end{array} \right.\end{equation}


\subsection{Bounds on the Ranks of $\ext_R^1(T_{i}/T_{i-1}, T_{i-1})$}

Let $T_n$ be a totally reflexive module over $(R, \fm, k)$ with $\fm^3=0$ that has an $n\times n$ upper triangular minimal presentation matrix. From Theorem 1.4 in \cite{BT}, we know that if there exists one nontrivial totally reflexive modules and if $k$ is of characteristic zero, then there are infinitely many non-isomorphic indecomposable totally reflexive modules of each admissible length, specifically a multiple of the embedding dimension of the ring. Assume that $R$ is also a finite dimension $k$-algebra. We investigate the rank of $\ext_R^1(T_{i}/T_{i-1}, T_{i-1})$ to get a sense of the complexity (or simplicity) of modules of each possible length. Consider the short exact sequence 
$$0 \rightarrow T_{i-1} \rightarrow T_{i} \rightarrow T_{i}/T_{-1} \rightarrow 0. \label{T2seq}$$

From this short exact sequence we can obtain a long exact sequence of Ext
$$ \cdots \rightarrow \ext_R^1(T_{i}/T_{i-1}, T_{i-1}) \rightarrow\ext_R^1(T_{i}/T_{i-1}, T_{i}) \rightarrow \ext_R^1(T_{i}/T_{i-1}, T_{i}/T_{i-1})\rightarrow \cdots.$$
We then  have the inequality \\
\begin{center} \small{
$\rightarrownk(\ext_R^1(T_{i}/T_{i-1}, T_{i})) \leq \rightarrownk(\ext_R^1(T_{i}/T_{i-1},T_{i}/T_{i-1})) + \rightarrownk(\ext_R^1(T_{i}/T_{i-1}, T_{i-1})).$\\}
\end{center}
Now we will find an upper bound for $ \Gamma(T_{i}/T_{i-1}, T_{i-1})$ over the ring $S=k[X, Y, Z]/(X^2, Y^2, Z^2, YZ).$   By (\ref{rank}), if we consider the case when $i=2,$ then we also have that $1 \leq \Gamma(T_{2}/T_1, T_{1})\leq 2$. Therefore
$$\rightarrownk\Gamma_S(T_{2}/T_1, T_{2}) \leq 4.$$
By induction we see that 
$$\Gamma_S(T_{i}/T_{i-1}, T_{i}) \leq 2n \qquad \mbox { for } i=2, 3 \ldots.$$

In fact, we know of cases when this rank is bounded below by one. That is, there exists a nontrivial short exact sequence beginning in $T_i$ and ending in $T_{i}/T_{i-1}.$  To find such a case, per \cite{BT}, we define the $b \times b$ upper triangular matrix
\begin{equation} \label{BTmatrix} M_b(s, t, u, v)= \left[\begin{array} {c c c c c c}
s & u& 0& 0& 0 & \dots\\
0& t & v& 0& 0&  \dots\\
0& 0& s & u& 0& \dots	\\
0& 0 & 0& t & v&  \\
0& 0 & 0&  0& s& \ddots  \\
\vdots& \vdots& \vdots& \vdots& &\ddots 
\end{array} \right],\end{equation}
and consider the following theorem.

\begin{theorem} \label{BTthm}\cite[Theorem 3.1]{BT}  Let $(R, \fm)$ be a local ring and assume that $ s$ and $ t$  are elements in $\fm\backslash \fm^2 $ that form an exact pair of zero divisors. Assume further that $ u $ and $ v $ are elements in $\fm\backslash \fm^2$ with $ uv=0$  and that one of the following conditions holds:

\noindent(a) The elements $s, t,$ and $ u$ are linearly independent modulo $\fm^2.$\\
(b) One has $ s\in (t) +\fm^2 $ and $ u, v \notin (t) + \fm^2.$

\noindent For every $b \in \mathbb{N}$, the $R$-module $M_b(s, t, u, v)$ is indecomposable, totally reflexive, and non-free. Moreover $\coker M_b(s, t, u, v)$ has constant Betti numbers, equal to $b.$ 
\end{theorem}

Over $S$ the all of the exact pairs of zero divisors are of the form $( x+\alpha y+ \beta z, x-\alpha y- \beta z) $ for $ \alpha,  \beta \in k$ \ref{expairsTR}. Now consider the choices for $u, v \in \fm \backslash \fm^2$ whose product is zero. From part (a) of Theorem \ref{BTthm}, $ x+\alpha y+ \beta z, x-\alpha y- \beta z $ and $u$ must be linearly independent in $\fm/\fm^2. $ For $  \gamma, \eta,  \lambda, \tau  \in k $ let $u=\gamma y+\eta z$  and $v= \lambda y + \tau z.$ Therefore, for $\gamma\neq \pm 2 \alpha$ and $\eta \neq \pm 2 \beta$ we have that $M_b(x+\alpha y+ \beta z, x-\alpha y- \beta z, \gamma y+\eta z, \lambda y + \tau z)$ is a presentation matrix of a nontrivial indecomposable totally reflexive $S$-module. 

Although Theorem \ref{BTthm} is useful in finding some of the indecomposable totally reflexive modules that have an upper triangular presentation matrix, it says nothing about whether the choices of $ s, t , u, $ and $ v$ will lead to non-isomorphic modules. In fact, notice that for the choices above of $u$ and $v$ over $S,$ neither one of them contains an $x$ term. This is because if either of them did, then one can always find an equivalent presentation matrix, and thus an isomorphic module, which does not have an $x$ term on the super diagonal. 


\section{Over Finite Fields}
Let us now consider the ring $\ZZ_2[X, Y, Z]/(X^2, Y^2, Z^2, YZ).$ Now we can list the four one-generated totally reflexive modules:
$$ (x),\quad (x+y),\quad  (x+z),\quad  (x+y+z),$$
all of which are non-isomorphic. From here we can construct all possible totally reflexive modules that have a $2 \times 2$ upper triangular presentation matrix, say
\begin{equation}\left[\begin{array}{c c}
u& a\\0&t \end{array} \right]. \label{2by2}
\end{equation}

However, there may be many modules that have isomorphic presentation matrices. Recall from \cite[Theorem 4.3]{equivmat}, that two modules are isomorphic if and only if they have equivalent presentation matrices. To discover which ones do, we start with the assumption that $a$ does not contain an $x$ term. Since $u$ and $t$ both must have an $x$ term,  we have that $\left[\begin{array}{c c}
u& a\\0&t \end{array} \right] $  is equivalent to $\left[\begin{array}{c c}
u& a-t\\0&t \end{array} \right]. $ Define 
$$\mathcal{T}=\{x, x+y, x+z, x+y+z\}\qquad \mbox{ and } \qquad \mathcal{ N}=\{y, z, y+z\}, \label {ordering}$$
and hence $u, t \in \mathcal{T}$ and $a \in \mathcal{N}.$
We will also impose an ordering on the elements in $\mathcal{T} $ and $\mathcal{ N}$ as the order listed above from smallest to largest. For $i=1,2$ let $u_i, t_i \in \mathcal{T}$ and $a_i \in \mathcal{N},$ the for two matrices
$$M_1=\left[\begin{array}{c c} u_1& a_1\\0&t_1 \end{array} \right] \qquad \mbox{ and } \qquad
 M_2=\left[\begin{array}{c c}u_2& a_2\\0&t_2 \end{array} \right]$$
 define $M_1$ to be \emph{smaller} than $M_2$ if
\begin{eqnarray*}
&&u_1<u_2 \mbox{ \: or }\\
&&u_1=u_2  \mbox{ \: and \,}t_1<t_2 \mbox{ \: or}\\
&&u_1=u_2,\, t_1=t_2 \mbox{ \: and } a_1<a_2.
\end{eqnarray*}
 Through the use of the CAS Magma, we can find the isomorphism classes of all totally reflexive modules that have a $2 \times 2$ upper triangular presentation matrix. We chose the smallest presentation matrix to represent each class. There are 24 non-isomorphic indecomposable totally reflexive modules with an upper triangular presentation matrix. 

In the below table, we list  representatives for each isomorphism class. For a matrix in the form of  (\ref{2by2}), the options for $a$ which represent non-isomorphic indecomposable totally reflexive modules are listed in the center of the table.

$$
\begin{array}{| l | | r | r | r | r |}
\hline 
u \; \downarrow \qquad t \ \rightarrow  & x & x+y & x+z & x+y+z\\
\hline \hline
x & y, z, y+z & z & y&y\\
\hline
x+y&  z& y, z, y+z &  y&y\\
\hline
x +z& y & y&  y, z, y+z & z\\
\hline
x +y+z& y & y& z&  y, z, y+z\\
\hline
\end{array}
$$

Something to note about these modules when the coefficient field is $\ZZ_2$ is that it is not possible to interchange $u$ and $t,$while keeping the same $a,$ and have them be isomorphic to each other. However, this is not the case if instead we consider modules over $\ZZ_3[x, y, z]/(x^2, y^2, z^2, yz).$  Over this ring, a module with presentation matrix
 $\left[\begin{array}{c c} u& a\\0&t \end{array} \right]$ is isomorphic to one with $\left[\begin{array}{c c}t& a\\0&u \end{array} \right]$ as a presentation matrix.

Theorem \ref{thm1} is useful in determining how the totally reflexive submodules of a totally reflexive module are contained in one another in this special case. However, there are totally reflexive modules which do not have an upper triangular presentation matrix and thus do not have a saturated TR-filtration. We conclude with an example of such a module.

\section{An example with no upper triangular presentation matrix}
\begin{example}\label{notUT}
Consider the $S$-module $T$ with a presentation matrix of 
$$\mathbf{T}:=\left[ \begin{array}{cc}
x &  z \\
y  &  x
\end{array} \right], $$
 where $S$ is the same ring defined previously, with char$(k)=0$. This is a totally reflexive module. Take a free resolution of it

$$\FF_T: \qquad \cdots \rightarrow S^2 \xrightarrow{\left[\begin{array}{c c} x&z\\y&x  \end{array}\right]}  S^2 \xrightarrow{\left[\begin{array}{l l}  x&-z\\-y&x  \end{array}\right]} S^2 \xrightarrow{\left[\begin{array}{c c} x&z\\y&x \end{array}\right]} S^2 \rightarrow 0$$
and apply the functor $\hm_S(\_,S)$.

$$  \hm_S(\FF_T, S): \quad 0 \rightarrow S^2 \xrightarrow{\left[\begin{array}{c c}  x&z\\y&x  \end{array}\right]} S^2 \xrightarrow{\left[\begin{array}{l l}  x&-z\\-y&x  \end{array}\right]} S^2 \xrightarrow{\left[\begin{array}{c c}  x&z\\y&x   \end{array}\right]}  S^2 \rightarrow \cdots$$

Therefore, $\hm_S(\FF_T, S)\cong \FF_T$ and hence the complex $\hm_S (\FF_T, S)$ is also  exact. This implies that $T$ is a totally reflexive $S$-module. In particular, this module is not isomorphic to a totally reflexive module that has an upper triangular presentation matrix. This can be done by showing that the matrix $\mathbf{T}$ is not equivalent to an upper triangular matrix.

\begin{proposition} For $\mathbf{T} =\left[\begin {array} {c c}
 x&z\\y&x \end{array} \right], $  a totally reflexive $S$-module,  the totally reflexive module $T = \coker(\mathbf{T})$ is not isomorphic to a totally reflexive module that has a upper triangular presentation matrix.
\end{proposition}

\begin{proof} Suppose that for an upper triangular matrix $\mathbf{U}$ we have $\coker \mathbf{T} \cong \coker\mathbf{ U},$
and hence $\mathbf{T}$ would be equivalent to $\mathbf{U}$. That is, there would exist  two invertible matrices  
$\left[ \begin {array} {c c} a&b\\c&d \end{array} \right]$ and $ \left[ \begin {array} {c c}  e&f\\g&h \end{array} \right] $ where $a, b, \ldots, h \in k$ such that

 \begin{eqnarray}
\left[ \begin {array} {c c} a&b\\c&d \end{array} \right] \;\mathbf{T} \; \left[ \begin {array} {c c}  e&f\\g&h \end{array} \right] = \mathbf{U}\label{product}. \end{eqnarray}
For $\alpha, \beta, \gamma, \zeta, \varphi, \psi, $ and $ \lambda $ in $k,$ let

$$ \mathbf{U}:= \left[ \begin {array} {c c} x+\alpha y + \beta z & \zeta x+ \varphi y + \psi z \\  0 & x+ \gamma y +\lambda z\end{array} \right]$$

\noindent Computing the products in (\ref{product}) gives us

\begin{eqnarray*}
\left[ \begin {array} {c c} a&b\\c&d \end{array} \right] 
 \left[ \begin {array} {c c} x&z\\y&x \end{array} \right ]
  \left[ \begin {array} {c c}  e&f\\g&h \end{array} \right] 
  &=& \left[ \begin {array} {c c} x+\alpha y + \beta z & \zeta x+ \varphi y + \psi z \\  0 & x+ \gamma y +\lambda z\end{array} \right] \\
  &&\\ &&\\
\small{\left[ \begin{array} {c c}(ae+bg)x +bey+agz & (af+bh)x + bfy +ah z\\ (ce+dg)x +edy+cgz & (cf-dh) x + dfy+chz  \end{array}\right]}
&=& \left[ \begin {array} {c c} x+\alpha y + \beta z & \zeta x+ \varphi y + \psi z \\  0 & x+ \gamma y +\lambda z\end{array} \right]. \\
\end{eqnarray*}

These yield a system of equations that one can show to be inconsistent. Therefore, $\mathbf{T}$ is not equivalent to $\mathbf{U}$ and thus $T$ is not isomorphic to another totally reflexive module that has a upper triangular presentation matrix.
\end{proof}
\end{example}

\subsection*{Acknowledgments}
I would like to thank David Jorgensen for his guidance in researching and writing this paper. I am also grateful to Graham Leuschke for his helpful comments and proofreading skills.

\end{document}